\documentclass[11pt,reqno]{amsart}

\usepackage[utf8]{inputenc}
\usepackage[T1]{fontenc}
\usepackage{amsmath, amssymb, amsthm}
\usepackage{setspace}
\usepackage{geometry}
\usepackage{enumitem}
\usepackage{hyperref}
\hypersetup{
  colorlinks,
  citecolor=blue,
  linkcolor=blue,
  urlcolor=green}
\usepackage{marginnote}
\usepackage{xcolor}
\usepackage{upgreek}
\usepackage{tikz}
\usetikzlibrary{calc}
\usepgflibrary {shadings} 
\usetikzlibrary{intersections}
\usetikzlibrary{patterns}
\usepackage{mathtools}
\usepackage{verbatim}
\usepackage{lipsum}  
\usepackage{mathrsfs}
\usepackage{stmaryrd}
\usepackage[
    colorinlistoftodos,
    backgroundcolor=green!20!white,
    bordercolor=green!75!black,
    linecolor=green!75!black
]{todonotes}
\usepackage{cleveref}
\usepackage{esint}
\numberwithin{equation}{section}

\theoremstyle{plain}
\newtheorem{theorem}{Theorem}
\newtheorem{lemma}{Lemma}[section]

\theoremstyle{definition}

\newtheorem{remark}{Remark}

\newcommand{\N}{\mathscr{N}}

\newcommand{\R}{\mathbf{R}}

\newcommand{\Sp}{\mathbf{S}}

\renewcommand{\div}{\operatorname{div}}

\newcommand{\Tr}{\operatorname{Tr}}
\newcommand{\G}{\mathcal{G}}
\newcommand{\pg}{\Xi_s}

\newcommand{\eps}{\varepsilon}

\begin{document}

\title[Uniform small energy regularity]{Uniform small energy regularity for fractional geometric problems}

\author{Marco Badran}
\address{ETH Z\"urich, Department of Mathematics, Rämistrasse 101, 8092 Zürich, Switzerland.}
 	\email{marco.badran@math.ethz.ch}
\author{Giacomo Cozzi}
\address{Dipartimento di Matematica ``Tullio Levi-Civita'', Universit\`a degli Studi di Padova, Via Trieste, 63, 35131 Padova PD, Italy}
\email{cozzi@math.unipd.it}

\begin{abstract}
    %We prove small energy regularity for a parabolic boundary reaction Ginzburg--Landau problem and for fractional harmonic maps to spheres, uniform for $s\in (s_0,1)$, for any $s_0\in (0,1)$. 
    %The first is a parabolic boundary reaction Ginzburg--Landau equation, where we show the validity of a small energy regularity in the case $s\in (0,1)\setminus\tfrac12$, which was open. Secondly, we consider fractional harmonic maps to spheres, where small energy regularity is known but we improve it 
    We prove small energy regularity for a parabolic boundary reaction Ginzburg--Landau problem in the full range $s\in (0,1)$, answering a question posed by Hyder, Segatti, Sire and Wang \cite{HyderSegattiSireWang}. We also obtain a similar small energy regularity result for fractional harmonic maps to spheres. Both results are uniform as $s\nearrow 1$.
\end{abstract}

\maketitle

%\tableofcontents

\section{Introduction}
Fractional harmonic maps were first introduced by Da Lio and Rivi\`ere \cite{DaLioRiv1,DaLioRiv2} as critical points of the $\dot H^{\frac12}(\R,\N)$ Gagliardo seminorm, providing a one-dimensional example of critical points of conformally invariant energies. Since then, they have attracted considerable attention as natural geometric objects, with applications including the Fraser--Schoen theory of extremal metrics for Steklov eigenvalues \cite{FraserSchoen16,FraserSchoen19}. Their regularity theory has been studied extensively, both in the extrinsic \cite{MillotPegon,MillotPegonSchikorra,MillotSireYu,KimNowakSire} and intrinsic \cite{Moser2011,Roberts} settings, while existence results have been obtained through Ginzburg--Landau type relaxations \cite{MillotSire,MillotSireWang}. Through the Caffarelli--Silvestre extension, fractional harmonic maps are also related to weighted harmonic maps with partially free boundary \cite{MillotPegonSchikorra,DuzaarSteffen,Hardt-Lin}. Alongside these elliptic questions, attention has also turned to the corresponding parabolic theory. This includes heat-flow formulations for half-harmonic maps, weak existence and partial regularity results, as well as related harmonic map flows with free boundary, where singularity formation phenomena arise \cite{SireWeiZheng,HyderSegattiSireWang,Struwe24}.

\medskip

More recently, fractional harmonic maps have found applications in the study of \emph{nonlocal minimal surfaces} in higher codimension, to which they are closely related \cite{SerraSeMA}. Roughly speaking, these nonlocal minimal surfaces can be described as the topological singular sets of fractional harmonic maps into spheres \cite{Badran26,CaselliFregugliaPicenni2024,CaselliFregugliaPicenni2026}. As a result, their regularity theory is tightly linked to that of fractional harmonic maps, via a suitable stratification procedure.

\medskip

In view of their role as nonlocal approximations to higher-codimension minimal surfaces as $s\nearrow 1$, it is essential to obtain regularity estimates for fractional harmonic maps whose constants are uniform in $s$, at least for $s$ sufficiently close to $1$. This is the main goal of the present paper.

\medskip

We present two results. \Cref{thm: small energy GL} is a small-energy regularity theorem for a parabolic boundary-reaction Ginzburg--Landau model, introduced in \cite{HyderSegattiSireWang} for $s=\frac12$ as a way of constructing a version of the fractional harmonic map flow. This flow, which corresponds to the gradient flow of the Caffarelli--Silvestre extension energy rather than of the Gagliardo seminorm, is perhaps less natural from an analytic point of view. For the purpose of finding stationary solutions, however, the two flows are equally effective, since they have the same stationary points. The advantage of the Caffarelli--Silvestre flow is that it enjoys a monotonicity formula, which is not available for the Gagliardo-seminorm flow, and this makes it more amenable to regularity arguments. 
In \cite{HyderSegattiSireWang} (see also \cite{GengWang}), a small energy regularity theorem for Ginzburg-Landau boundary reactions was proved when $s=\frac12$ with techniques that, according to the authors, do not generalise to other values of $s$. We extend the $\eps$-regularity result of \cite{HyderSegattiSireWang} to the full range $s\in (0,1)$, with estimates that are uniform for $s\in (s_0,1)$, for any $s_0\in (0,1)$. 

\medskip

The second result, \Cref{thm: small energy FHM}, concerns fractional harmonic maps through their closely related formulation as weighted harmonic maps with partially free boundary; see \cite[Proposition 3.13]{MillotPegonSchikorra}. We describe a general strategy for upgrading an $\eps$-regularity result valid for each fixed $s\in(0,1)$ to a \emph{uniform} $\eps$-regularity result in the range $s\in(s_0,1)$. As a direct corollary, building on the results of \cite{MillotPegonSchikorra}, we establish uniform small-energy regularity when the target is a sphere, the case relevant to nonlocal minimal surfaces.

\medskip

The main ingredient in the proof of \Cref{thm: small energy GL} is a Bochner inequality for boundary reactions, combined with a Harnack inequality that is uniform as $s\nearrow 1$. Despite several modifications, the underlying strategy follows the approach of Struwe \cite{Struwe88}, which in turn builds on an argument of Schoen \cite{Schoen84} in the stationary case. The proof of \Cref{thm: small energy FHM} is based on a similar idea. The main difference is that Schoen's argument requires qualitative smoothness a priori: this is available for Ginzburg--Landau boundary reactions, but not for stationary harmonic maps. We therefore first prove quantitative estimates, under a small-energy assumption, for smooth solutions; these estimates are then used to establish qualitative smoothness under the same small-energy condition, via a dimension-reduction argument.

\subsection{Main results}

Denote by $X=(x,z)\in\R^{m+1}_+\coloneqq \R^{m}\times (0,+\infty)$  and, for $s\in (0,1)$, let $a\coloneqq 1-2s$. Fixing a smooth datum $U_0\colon \R^{m+1}_+\to \R^\ell$ and a number $\eps>0$, we consider the evolution problem 
\begin{equation}\label{eq: GL flow}
	\begin{cases}
		z^a\partial_tU=\div(z^a\nabla U)&\text{in }\R^{m+1}_+\times (0,+\infty),\\
		U =U_0&\text{on }\R^{m+1}_+\times \{0\}\\
		-\partial_z^a\vert_{z=0}U=\eps^{-2s}(1-|U|^2)U&\text{in }\R^{m}\times (0,+\infty)
	\end{cases}
\end{equation}
where we used the shorthand 
\begin{equation}\label{eq: dza}
    \partial_z^a\vert_{z=0}U(x)\coloneqq \lim_{z\downarrow 0}z^a\partial_zU(x,z)
\end{equation}
and the differential operators $\div$ and $\nabla$ are understood with respect to the extended variable $X=(x,z)$. 

\medskip

Problem \eqref{eq: GL flow} is a parabolic Ginzburg--Landau boundary reaction system, corresponding to the local extension of the nonlocal flow 
\begin{equation}\label{eq: nonlocal GL flow}
    \begin{cases}
        (\partial_t-\Delta)^su=\eps^{-2s}(1-|u|^2)u&\text{in }\R^m\times (0,+\infty)\\
        u(x,t)=u_0&\text{in }\R^m\times (-\infty,0]
    \end{cases}
\end{equation}
where $(\partial_t-\Delta)^s$ is the Fourier multiplier with symbol $((2\pi|\xi|)^2+2\pi i\sigma)^s$, being $(\xi,\sigma)$ the frequency variables in $\R^m\times \R$. The correspondence between \eqref{eq: GL flow} and \eqref{eq: nonlocal GL flow} is explained extensively in \cite[Section 2.1]{HyderSegattiSireWang}.  In turn, \eqref{eq: nonlocal GL flow} is often used as an approximation of the fractional harmonic map heat flow to spheres
\begin{equation}\label{eq: nonlocal HM flow}
    \begin{cases}
        (\partial_t-\Delta)^su\perp T_u\Sp^\ell&\text{in }\R^m\times (0,+\infty)\\
        u(x,t)=u_0&\text{in }\R^m\times (-\infty,0].
    \end{cases}
\end{equation}
Problem \eqref{eq: GL flow} corresponds to the gradient flow of the Caffarelli--Silvestre extension energy \cite{CaffarelliSilvestre} with an extra potential term in the thin space (see \eqref{eq: CS with GL} below). While this approach may seem at first less natural than the gradient flow of the standard Gagliardo $\dot H^s$ seminorm, corresponding to the operator $\partial_t+(-\Delta)^s$ rather than $(\partial_t-\Delta)^s$, it is just as effective to find critical points as long-time limit of solutions. The striking advantage of it is that, as opposed to the Gagliardo seminorm flow, the operator $(\partial_t-\Delta)^s$ has a monotonicity formula (see \cite[Lemma 3.1]{HyderSegattiSireWang} or \Cref{lem: monotonicity}), which opens the doors to proving (partial) regularity of the flow. Denoting by $\G^s_{Z_0}$ the fundamental solution at a space-time point $Z_0=(X_0,t_0)\in \partial\R^{m+1}_+\times (0,\infty)$, the quantity 
\begin{equation*}
\begin{split}
    \Phi_\eps (U,Z_0,R)&\coloneqq  \int_{t_0-4R^2}^{t_0-R^2}\int_{\R^{m+1}_+} z^a\frac{|\nabla U|^2}{2}\mathcal{G}^s_{Z_0}(X,t)\,dXdt\\
    &+\int_{t_0-4R^2}^{t_0-R^2}\int_{\R^m}\frac{(1-|U|^2)^2}{\eps^{2s}}\mathcal{G}^s_{Z_0}(x,t)dxdt,
\end{split}
\end{equation*}
is monotone non-decreasing. Our first theorem is a small energy quantitative regularity for solutions to \eqref{eq: GL flow}.

\begin{theorem}\label{thm: small energy GL}
    For any $s_0\in (0,1)$ there exist positive constants $\eta_0,\delta$ and $C$ depending only on $m,\ell$ and $s_0$ such that the following holds. Let $U$ be a solution to \eqref{eq: GL flow} for some $s\in (s_0,1)$ such that, for some $Z_0\in \partial\R^{m+1}_+\times (0,+\infty)$,
    \begin{equation*}
        \Phi_\eps(U,Z_0,R)\leq \eta_0,
    \end{equation*}
    then 
    \begin{equation*}
        \sup_{P_{\delta R}(Z_0)}R^2\left(|\nabla_xU|^2+|\partial_tU|\right)\leq C\delta^{-2}.
    \end{equation*}
    where $P_{\delta R}(Z_0)$ is a parabolic cylinder of size $\delta R$.
\end{theorem}
\begin{remark}
    Observe that solutions to \eqref{eq: GL flow} are qualitatively regular by standard parabolic theory for boundary reactions. The point of \Cref{thm: small energy GL} is the quantification of the bounds, plus their uniformity over $(s_0,1)$. This result was already proved for $s=\tfrac12$ in \cite[Lemma 4.3]{HyderSegattiSireWang}.
\end{remark}

The second result we prove is a uniform small energy regularity result for fractional harmonic maps to spheres, that is maps $u\colon \R^{m}\to\Sp^\ell$ weakly satisfying 
\begin{equation}\label{eq: FHM}
    (-\Delta)^su\perp T_u\Sp^{\ell},\quad \text{in }\R^m
\end{equation}
corresponding formally to (traces of) critical points of 
\begin{equation*}
    E_s(U)\coloneqq \int_{\R^{m+1}_+}z^a|\nabla U|^2dX
\end{equation*}
among maps $U\in H^1_a(\R^{m+1}_+,\R^\ell)$ satisfying $(\Tr U)(x)\in \Sp^\ell$ for a.e.~ $x\in\R^{m}$. As a boundary reaction, \eqref{eq: FHM} reads as
\begin{equation}\label{eq: FHM br}
	\begin{cases}
		\div(z^a\nabla U)=0&\text{in }\R^{m+1}_+,\\
        U(x,0)\in\Sp^\ell&\text{a.e.~ }x\in \R^{m}
        \\
		-\partial_z^a\vert_{z=0}U(x)=|d_su|^2(x)u(x)&\text{in }\R^{m}
	\end{cases}
\end{equation}
where $\partial_z^a\vert_{z=0}$ is defined as in \eqref{eq: dza} and $|d_su|^2$ is defined in \eqref{eq: frac grad} below. 

\medskip

It is known that local minimisers and stationary critical points satisfy the monotonicity formula 
\begin{equation}\label{eq: FHM monotonicity}
    \frac{d}{dr}\Psi_s(U,x_0,r)\geq 0,\quad \text{where}\quad \Psi_s(U,x_0,r)\coloneqq r^{m-2s}\int_{B_r^+(x_0)}z^a|\nabla U|^2dX
\end{equation}
where $x_0\in \partial\R^{m+1}_+$ and $B_r^+(x_0)\coloneqq B_r(x_0)\cap\{z\geq 0\}$.

\medskip

Our second result is a uniform small energy regularity theorem for fractional harmonic maps to spheres.

\begin{theorem}\label{thm: small energy FHM}
    For any $s_0\in (0,1)$ there exist positive constants $\eps_0,\delta$ and $C$ depending only on $m,\ell$ and $s_0$ such that the following holds. Let $U$ be a solution to \eqref{eq: FHM br} satisfying \eqref{eq: FHM monotonicity} for some $s\in (s_0,1)$. If, for some $x_0\in \R^{m}$ and $r>0$ 
    \begin{equation*}
        \Psi_s(U,x_0,r)\leq \eps_0
    \end{equation*}
    then 
    \begin{equation*}
        \sup_{B_{\delta r}} r^2|\nabla u|^2\leq C\delta^{-2}
    \end{equation*}
    where $u\coloneqq \Tr U$.
\end{theorem}
We remark that the $\eps$-regularity theorem for fractional harmonic maps to spheres was already proved in \cite[Theorem 5.1]{MillotPegonSchikorra} but the general procedure outlined there does not ensure uniformity of the parameters as $s\nearrow 1$. We introduce a general procedure to upgrade the existing small energy regularity to uniform small energy regularity for fractional harmonic maps. 

\subsection{The extension method}
We always denote $X=(x,z)\in \R^{m+1}_+\eqqcolon \R^{m}\times (0,+\infty)$. We identify $\R^m\simeq \partial\R^{m+1}_+$ by identifying $x\simeq (x,0)$. For any $s\in (0,1)$, we define $a\coloneqq 1-2s\in (-1,1)$ and define the Caffarelli--Silvestre extension energy 
\begin{equation}\label{eq: CS}
    E_s(U)\coloneqq \frac{\delta_s}{2}\int_{\R^m}z^a|\nabla U|^2dX
\end{equation}
on the space
\begin{equation*}
    H^1_a(\R^m,\R^\ell)\coloneqq \left\lbrace U\in L^2_{\mathrm{loc}}(\R^m,\R^\ell):\nabla U\in L^2(z^adX)\right\rbrace.
\end{equation*}
where 
\begin{equation*}
    \delta_s\coloneqq 2^{2s-1}\frac{\Gamma(s)}{\Gamma(1-s)}.
\end{equation*}
Energy \eqref{eq: CS} was introduced in \cite{CaffarelliSilvestre} as a way of representing nonlocal problems as local problems in an extended space. Minimising $E_s$ over the class of functions $U\in H^1_a$ with prescribed trace $\Tr U=u$ yields a weak solution to the problem 
\begin{equation}\label{eq: CS extension}
    \begin{cases}
        \div(z^a\nabla U)=0&\text{in }\R^{m+1}_+\\
        U(x,0)=u(x)&\text{on }\R^m
    \end{cases}
\end{equation}
whose weighted Dirichlet-to-Neumann map gives the fractional laplacian of the datum 
\begin{equation}\label{eq: D2N}
    -\lim_{z\downarrow0}z^a\partial_z U(x,z)=(-\Delta)^su(x).
\end{equation}
The gradient flow associated to \eqref{eq: CS} is
\begin{equation}\label{eq: GFCS}
    \begin{cases}
        z^a\partial_tU=\div(z^a\nabla U)&\text{in }\R^{m+1}_+\\
        U(x,0)=u(x)&\text{on }\R^m
    \end{cases}
\end{equation}
whose Dirichlet-to-Neumann operator gives 
\begin{equation}\label{eq: D2N par}
    -\lim_{z\downarrow0}z^a\partial_z U(x,z)=(\partial_t-\Delta)^su(x).
\end{equation}
The operators in the right-hand sides of \eqref{eq: D2N} and \eqref{eq: D2N par} are Fourier multiplier operators with symbols, 
\begin{equation*}
    (2\pi|\xi|)^{2s}\quad\text{and}\quad  ((2\pi|\xi|)^{2}+2\pi i\sigma)^s
\end{equation*}
which can be represented as integro-differential operators as follows (see, e.g., \cite{MillotPegonSchikorra,HyderSegattiSireWang}),
\begin{equation}\label{eq: frac lapl}
    (-\Delta)^su(x)\coloneqq \int_{\R^m}\left(u(x)-u(y)\right)K_s(x-y)dy
\end{equation}
and 
\begin{equation}\label{eq: frac heat}
     (\partial_t-\Delta)^su(x,t)\coloneqq \int_0^{+\infty}\int_{\R^m}\left(u(x,t)-u(x-z,t-\tau)\right)H_s(z,\tau)dzd\tau 
\end{equation}
where the kernels are given, respectively, by 
\begin{equation}
\label{kernels}
    K_s(z)\coloneqq s2^{2s}\frac{\Gamma\left(\frac{m+2s}{2}\right)}{\pi^\frac{m}{2}\Gamma(1-s)}\frac{1}{|z|^{m+2s}},\quad H_s(z,\tau)\coloneqq \frac{1}{(4\pi)^{\frac{m}{2}}|\Gamma(-s)|}\,
\frac{e^{-\frac{|z|^2}{4\tau}}}{\tau^{\frac{m}{2}+1+s}}
\end{equation}
where $z\in\R^m$, $\tau>0$ and the integrals are taken in the principal value sense. We also set 
\begin{equation}\label{eq: frac grad}
    |d_su|^2(x)=\frac{c_{m,s}}{2}\int_{\R^m}\frac{|u(x)-u(y)|^2}{|x-y|^{m+2s}}dy,\qquad c_{m,s}\coloneqq  s2^{2s}\frac{\Gamma\left(\frac{m+2s}{2}\right)}{\pi^\frac{m}{2}\Gamma(1-s)}
\end{equation}
We also record the backwards fundamental solution of \eqref{eq: GFCS} about a space-time point $(X_0,t_0)$
\begin{equation}
	\G^s_{(X_0,t_0)}(X,t)\coloneqq \frac1{\Gamma(s)(4\pi)^{\frac{m}2}}\frac{e^{-\frac{|X-X_0|^2}{4(t_0-t)}}}{|t-t_0|^{\frac{m}{2}+1-s}},\qquad t<t_0.
\end{equation}
and denote $\G^s\coloneqq \G^s_{(0,0)}$.

\medskip

Unless otherwise specified, we will always use lowercase letters $u,v,\dots$ to denote thin-space traces of functions $U,V,\dots$ in the extended space.

\section{Parabolic Ginzburg--Landau boundary reactions}
We define the universal parameters of the problem to be the positive integers $m,\ell$ and any $s_0\in (0,1)$. We introduce the notation $A\lesssim B$ to say that $A\leq CB$ for some constant $C>0$ depending only on the universal parameters.  

\medskip

We start by considering solutions $U$ to \eqref{eq: GL flow}. We remark that it is always possible to construct solutions to \eqref{eq: GL flow}, for example via a discrete approximation scheme, see \cite[Section 2]{HyderSegattiSireWang}.

\medskip

Equation \eqref{eq: GL flow} is the gradient flow associated to the boundary reaction Ginzburg--Landau energy functional 
\begin{equation}\label{eq: CS with GL}
	F_s(V)\coloneqq \int_{\R^{m+1}_+}z^a\frac{|\nabla V|^2}2dX+\frac{1}{4\eps^{2s}}\int_{\R^m}(1-|v|^2)^2dx
\end{equation}
where $v=\Tr V$. We also denote 
\begin{equation}\label{eq: E0}
E_0\coloneqq \int_{\R^{m+1}_+}z^a|\nabla U_0|^2dX    
\end{equation}
and recall that, along solutions of \eqref{eq: GL flow} 
\begin{equation*}
    \int_{\R^{m+1}_+}z^a|\nabla U(\cdot,t)|^2dX\leq E_0
\end{equation*}
by \cite[Formula (2.33)]{HyderSegattiSireWang}.

\subsection{The monotone quantity}
Let $Z_0\coloneqq (X_0,t_0)$ be a fixed space-time point. For any $R>0$, define the extended time slab
\begin{equation*}
    T_R^+(Z_0)\coloneqq \{(X,t)\in \R^{m+1}_+\times(0,+\infty):-4R^2<t-t_0<-R^2\}
\end{equation*}
and, if $X_0\in\R^m$, the thin time  slab
\begin{equation*}
    T_R(Z_0)\coloneqq \{(x,t)\in \R^{m}\times(0,+\infty):-4R^2<t-t_0<-R^2\}.
\end{equation*}
Whenever $t_0=0$, we simply denote the above sets by $T_R^+$ and $T_R$. For any $U\in C^\infty (\R^{m+1}_+\times (0,+\infty))$, any $Z_0=(X_0,t_0)$ with $X_0\in \R^m,\, t_0>0$ and any $R>0$, we set
\begin{equation*}
	\Phi_\eps (U,Z_0,R)\coloneqq  \int_{T_R^+} z^a\frac{|\nabla U|^2}{2}\mathcal{G}^s_{Z_0}(X,t)\,dXdt+\int_{T_R }W(u)\mathcal{G}^s_{Z_0}(x,t)dxdt,
\end{equation*}
where $W(u)=\frac{(1-|u|^2)^2}{4\eps^{2s}}$. On solutions to \eqref{eq: GL flow}, we have the following monotonicity result.

\begin{lemma}[Monotonicity formula, 
{\cite[Lemma 3.1]{HyderSegattiSireWang}}]\label{lem: monotonicity}
	For any $Z_0=(X_0,t_0)$ with $X_0 \in\R^{m}$ and $t_0>0$, if $U$ solves \eqref{eq: GL flow}, then the quantity $\Phi_{\eps}(U,Z_0,R)$ is non-decreasing with respect to $R$ in $(0,\tfrac12\sqrt{t_0})$.
\end{lemma}
The following is a statement about monotonicity over time slices with the precise remainder term. We remark that the monotonicity statement alone, without the explicit expression for the remainder term, already appears in \cite[Lemma 3.1]{HyderSegattiSireWang}.

\begin{lemma}\label{lem: monotonicity remainder}
	Let $U$ be a solution to \eqref{eq: GL flow} and let $Z_0=(X_0,t_0)\in \R^{m}\times (0,+\infty)$. Then for every $t<t_0$ the relation 
	\begin{equation}\label{eq: monotonicity remainder}
		\begin{split}
			\frac{d}{dt}&\Bigg((t-t_0)\int_{\R^{m+1}_+}z^a\frac{|\nabla U|^2}{2}\G^s_{Z_0}+(t-t_0)\int_{\R^m}W(u)\G^s_{Z_0}\Bigg)\\
			&=\frac{1}{4|t-t_0|}\int_{\R^{m+1}_+}z^a|2(t-t_0)\partial_tU+\nabla U\cdot (X-X_0)|^2\G^s_{Z_0}+\int_{\R^m}W(u)\G^s_{Z_0}
		\end{split}
	\end{equation}
	holds.
\end{lemma}
\begin{proof}
	Up to a translation we can assume $Z_0=(0,0)$. In the computations below, we use 
    \begin{equation*}
        z^a\partial_t\G^s+\div(z^a\nabla\G^s)=0,\quad \text{and}\quad \nabla\G^s=\frac{X}{2t}\G^s,\qquad \forall X\in \R^{m+1}_+,\ t<0.
    \end{equation*}
We start by computing 
\begin{equation*}
\begin{split}
	\frac{d}{dt}\left(t\int_{\R^{m+1}_+}z^a\frac{|\nabla U|^2}{2}\G^s\right)&=\int_{\R^{m+1}_+}z^a\frac{|\nabla U|^2}{2}\G^s+t\int_{\R^{m+1}_+}z^a\nabla U\cdot\nabla\partial_tU\G^s\\
	&+t\int_{\R^{m+1}_+}z^a\frac{|\nabla U|^2}{2}\partial_t\G^s
\end{split}
\end{equation*}
The second term becomes, after an integration by parts
\begin{equation*}
	\begin{split}
		t\int_{\R^{m+1}_+}z^a\nabla U\cdot\nabla\partial_tU\G^s&=-t\int_{\R^{m+1}_+}\div(z^a\nabla U)\partial_tU\G^s-t\int_{\R^{m+1}_+}z^a(\nabla U\cdot\nabla\G^s)\partial_tU\\
		&+t\int_{\R^m}(-\partial_z^a\vert_{z=0}U)\partial_tU\G^s\\
		&=-t\int_{\R^{m+1}_+}z^a|\partial_tU|^2\G^s-\frac12\int_{\R^{m+1}_+}z^a\partial_tU(\nabla U\cdot X)\G^s\\
		&-t\int_{\R^m}W'(u)\partial_tu\G^s
	\end{split}
\end{equation*}
We also have
\begin{equation*}
	\begin{split}
		t\int_{\R^{m+1}_+}z^a\frac{|\nabla U|^2}{2}\partial_t\G^s&=-t\int_{\R^{m+1}_+}\frac{|\nabla U|^2}{2}\div(z^a\nabla\G^s)\\
		&=t\int_{\R^{m+1}_+}z^aD^2U\left[\nabla U,\nabla \G^s\right]\\
		&=\frac12\int_{\R^{m+1}_+}z^aD^2U\left[\nabla U,X\right]\G^s\\
		&=\frac12\int_{\R^{m+1}_+}z^a\left(\nabla(\nabla U\cdot X)\cdot \nabla U-|\nabla U|^2\right)\G^s
	\end{split}
\end{equation*}
where the boundary term vanishes because
\begin{equation*}
	\lim_{z\to 0}z^a\partial_z\G^s=\lim_{z\to 0}z^{a+1}\G^s=0.
\end{equation*}
Integrating again by parts, we obtain
\begin{equation*}
	\begin{split}
		\frac12\int_{\R^{m+1}_+}z^a\nabla(\nabla U\cdot X)\cdot \nabla U\G^s&=-\frac12\int_{\R^{m+1}_+}\div(z^a\nabla U)(\nabla U\cdot X) \G^s\\
		&-\frac12\int_{\R^{m+1}_+}z^a(\nabla U\cdot X)\nabla U\cdot\nabla\G^s\\
		&+\frac12\int_{\R^{m+1}_+}(-\partial_z^a\vert_{z=0}U)(\nabla U\cdot X)\G^s\\
		&=-\frac12\int_{\R^{m+1}_+}z^a\partial_tU(\nabla U\cdot X) \G^s-\frac1{4t}\int_{\R^{m+1}_+}z^a(\nabla U\cdot X)^2\\
		&-\frac12\int_{\R^{m+1}_+}W'(u)(\nabla u\cdot x)\G^s.
	\end{split}
\end{equation*}
Lastly, we compute 
\begin{equation*}
	\begin{split}
		\frac{d}{dt}\left(t\int_{\R^m}W(u)\G^s\right)&=\int_{\R^m}W(u)\G^s+t\int_{\R^m}W'(u)\partial_tu\G^s+t\int_{\R^m}W(u)\partial_t\G^s
	\end{split}
\end{equation*}
and integrating by parts we have 
\begin{equation*}
	t\int_{\R^m}W(u)\partial_t\G^s=\frac12\int_{\R^{m+1}_+}W'(u)(\nabla u\cdot x)\G^s.
\end{equation*}
Putting everything together and recognising the square 
\begin{equation*}
	|2t\partial_tU+\nabla U\cdot X|^2=4t^2|\partial_tU|^2+4tU(\nabla U\cdot X)+(\nabla U\cdot X)^2
\end{equation*}
the conclusion follows.
\end{proof}

\subsection{The Bochner inequality}
The following result is a Bochner-type inequality for parabolic Ginzburg--Landau boundary reactions. 
%This, combined with a suitable Harnack inequality and a classical argument by Schoen [cite] yields will 
\begin{lemma}[Bochner inequality for boundary reactions]\label{lem: Bochner}
	Let $U$ be a smooth solution to \eqref{eq: GL flow}. Then the density 
	\begin{equation*}
		e(U)\coloneqq |\partial_t U|+|\nabla_x U|^2
	\end{equation*}
	satisfies 
	\begin{equation*}
		\begin{cases}
			z^a\partial_te(U)-\div(z^a\nabla e(U))\leq 0&\text{in }\R^{m+1}_+\times (0,+\infty)\\
			-\partial_z^a\vert_{z=0}e(U)\leq 2\eps^{-2s}(1-|U|^2)e(U)&\text{on }\R^m\times (0,+\infty).
		\end{cases}
	\end{equation*}
	\begin{proof}
		We compute, with the shorthand $\partial_j\coloneqq \partial_{x_j},$
		\begin{equation*}
		\begin{split}
			z^a\partial_t(|\nabla_xU|^2)-\div(z^a\nabla (|\nabla_xU|^2))&=2\sum_{j=1}^m\partial_jU\cdot (z^a\partial_t\partial_jU-\div(z^a\nabla \partial_jU))\\
			&-2\sum_{j,k=1}^mz^a|\partial_{jk}U|^2-2\sum_{j=1}^mz^a|\partial_{jz}U|^2
		\end{split}
		\end{equation*}
		and using that 
		\begin{equation*}
			z^a\partial_t\partial_jU-\div(z^a\nabla \partial_jU)=\partial_j(z^a\partial_tU-\div(z^a\nabla U))=0
		\end{equation*}
		we find that the above quantity is $\leq 0$. Similarly, 
		\begin{equation*}
			\begin{split}
			z^a\partial_t(|\partial_tU|)-\div(z^a\nabla (|\partial_tU|))&=\frac{\partial_tU}{|\partial_tU|}\cdot \partial_t(z^a\partial_tU-\div(z^a\nabla U))-\sum_{j=1}^{m+1}z^a\partial_{jt}U\partial_j\left(\frac{\partial_{t}U}{|\partial_{t}U|}\right)
		\end{split}
		\end{equation*}
        where we denoted, for brevity $\partial_{m+1}\coloneqq \partial_z$. Recognizing the projection 
        \begin{equation*}
            \partial_{jt}U\partial_j\left(\frac{\partial_{t}U}{|\partial_{t}U|}\right)
            %=\frac{1}{|\partial_t U|}\left(|\partial_{jt}U|^2-\frac{(\partial_t U \cdot \partial_{jt}U)^2}{|\partial_t U|^2}\right)
=
\frac{1}{|\partial_t U|}
\left|
\partial_{jt}U
-
\frac{\partial_t U \cdot \partial_{jt}U}{|\partial_t U|^2}\,\partial_t U
\right|^2
\ge 0.
        \end{equation*}
		we obtain the nonpositivity of the above term, and thus
		\begin{equation*}
			z^a\partial_te(U)-\div(z^a\nabla e(U))\leq 0.
		\end{equation*}
		The second inequality follows from 
		\begin{equation}\label{eq: neumann GL}
	\begin{split}
		-\partial_z^a\vert_{z=0}e(U)&=|\partial_tU|^{-1}\partial_tU\cdot \partial_t(\eps^{-{2s}}(1-|U|^2)U)+2\nabla_xU\cdot \nabla_x(\eps^{-{2s}}(1-|U|^2)U)\\
		&=\eps^{-{2s}}(1-|U|^2)(|\partial_tU|+2|\nabla_x U|^2)\\
        &-2\eps^{-{2s}}(U\cdot \nabla_x U)^2-2\eps^{-{2s}}|\partial_tU|^{-1}(U\cdot \partial_t U)^2\\
		&\leq 2\eps^{-{2s}}(1-|U|^2)e(U)
	\end{split}
\end{equation}
on $\{z=0\}$.
	\end{proof}
\end{lemma}

Given any point $X_0\in\R^m$ and a time $t_0>0$ we define the (extended) parabolic cylinder of size $r>0$ 
\begin{equation*}
    P^+_r(X_0,t_0)\coloneqq\left\lbrace (X,t)\in \R^{m+1}_+\times (0,+\infty): |X-X_0|< r,\ |t-t_0|< r^2 \right\rbrace
\end{equation*}
and $P_r(X_0,t_0)\coloneqq P_r^+(X_0,t_0)\cap\{z=0\}$.
The next ingredient we need is a simple ``clearing-out'' lemma for small energy solutions of \eqref{eq: GL flow}. We remark that the case $s=\frac12$ is already proved in \cite[Lemma 4.2]{HyderSegattiSireWang}.

\begin{lemma}[Clearing-out]\label{lem: clearing out}
	There exist positive constants  $\eta_1,\delta$ and $\eps_0$ depending only on $m$ and $s_0$ such that the following holds. If $Z_0\in \R^m\times(0,+\infty)$ and  $U$ is a smooth solution to  \eqref{eq: GL flow} for $\eps\leq \eps_0$ satisfying 
	\begin{equation*}
		\Phi_\eps (U,Z_0,1)\leq \eta_1,
	\end{equation*}
	then 
	\begin{equation*}
		|U|\geq \frac{1}{\sqrt{2}}\quad \text{on }P_{\delta}(Z_0).
	\end{equation*}
\end{lemma}
\begin{proof}
    The proof is standard and we sketch it for completeness. The conclusion follows by the fact smallness of $|U|$ forces a positive lower bound on the potential term in $\Phi_\eps$. 

    \medskip

    Without loss of generality, we assume $(X_0,t_0)=(0,0)$. By assumption and the monotonicity formula we have, for any $r\in (0,1)$,
        \begin{equation}
            \int_{-4r^2}^{-r^2}\int_{\R^m}\frac1{\eps^{2s}}(1-|u|^2)^2\G^s(x,t)dxdt<\eta_1.
        \end{equation}
        Suppose by contradiction that for every $\delta>0$ there is a point $X_*=(x_*,0)\in \R^m$ and $t_*>0$ such that $Z_*\coloneqq (X_*,t_*)\in P_\delta$ and $|U(Z_*)|<\tfrac{1}{\sqrt{2}}$. The rescaling $V(X,t)\coloneqq U(X_*+\eps X,t_*+\eps^2 t)$ satisfies \eqref{eq: GL flow} with $\eps=1$, thus by regularity theory there is $\delta_1>0$ such that 
        \begin{equation*}
            |V|\leq \frac{3}{4},\quad\text{in }P_{\delta_1}^+,
        \end{equation*}
        which means, scaling back, that 
        \begin{equation*}
            \frac{(1-|u|^2)^2}{\eps^{2s}}\geq c_1\eps^{-2s},\quad \text{on }P_{\delta_1\eps}
        \end{equation*}
        for some $c_1>0$. The Gaussian centred in $Z_*$ satisfies
        \begin{equation*}
            \G_{Z_*}^s(x,t)\geq c_2\eps^{-m-2+2s},\quad \text{in }\{(X,t):|X-X_*|\leq \eps,\, -4\eps^2\leq t-t_*\leq -\eps^2\}
        \end{equation*}
        getting the lower bound 
        \begin{align*}
            \Phi_\eps(U,Z_*,1)&\geq \Phi_{\eps}(U,Z_*,\eps)\\
            &\geq \int_{T_\eps(Z_*)}\frac{(1-|u|^2)^2}{\eps^{2s}}\G^s_{Z_*}(x,t)dxdt\\
            &\geq c_1c_2\eps^{2s}\eps^{m+2}\eps^{-m-2*2s}\\
            &=c_1c_2>0.
        \end{align*}
        The conclusion follows by 
    approximating the Gaussian in $Z_*$ with the Gaussian at the origin by taking $\delta$ sufficiently small (exactly as in \Cref{lem: smallness} below), yielding a contradiction.
\end{proof}

The clearing out under small energy ensures that the term $(1-|u|^2)/\eps^{2s}$ appearing in the boundary reaction of the Bochner inequality can be controlled by $e(U)$ itself, according to the following result.

\begin{lemma}\label{lem: control pot by e} Under the same assumptions of \Cref{lem: clearing out}, we have 
	\begin{equation*}
		\sup_{P_{1/2}(Z_0)}\frac{(1-|u|^2)}{\eps^{2s}}\lesssim \|e(U)\|^2_{L^\infty(P_1^+)}+1.
	\end{equation*}
\end{lemma}
\begin{proof}
    Using the representation formula \eqref{eq: frac heat}, we readily compute
    \begin{equation}\label{eq: eq for |u|}
        (\partial_t-\Delta)^s \frac{|u|^2}{2}=u(\partial_t-\Delta)^su-|\pg u|^2,
    \end{equation}
    where
    \begin{equation*}
        \pg u(t,x)=\left(\frac{1}{2}\int_0^{\infty}\int_{\R^m}(u(t,x)-u(t-\tau,x-z))^2 H_s(\tau, z)\,dz\,d\tau\right)^{\frac{1}{2}},
    \end{equation*}
    and $H_s$ is defined as in \eqref{kernels}.
    We now show that $|\pg u|^2$ can be bounded by means of 
    $\|\partial_t u\|_{\infty}$, 
    $\|\nabla u\|_{\infty}$. 
    Fix a space-time point $Z_0\coloneqq (X_0,t_0)\in \R^m\times (0,+\infty)$; then, for any 
    $(x,t)\in P_{1/2}(Z_0)$
    , it holds
    \begin{equation*}
    \begin{split}
    |\pg u|^2(x,t)&=\frac{1}{2}\int_{P_1(Z_0)}(u(t,x)-u(t-\tau,x-z))^2 H_s(\tau, z)\,dz\,d\tau\\&+\frac{1}{2}\int_{(P_1(Z_0))^c}(u(t,x)-u(t-\tau,x-z))^2 H_s(\tau, z)\,dz\,d\tau.
    \end{split}
    \end{equation*}
    Expanding, we get 
    \begin{equation*}
        (u(t,x)-u(t-\tau,x-z))^2\lesssim \|\partial_tu\|_{L^\infty(P_1(Z_0))}^2\tau^2+\|\nabla_xu\|_{L^\infty(P_1(Z_0))}^2|z|^2
    \end{equation*}
    for every $(x,t)\in P_{1/2}(Z_0)$ and every $(z,\tau)\in P_{1}(Z_0)$.
    Thus, using that $|u|\leq 1$ a.e.~ and noting that $H_s$ is integrable away from the origin, we obtain for any $(x,t)\in P_{1/2}(Z_0)$,
    \begin{equation*}
    \begin{split}
        |\pg u|^2(x,t)&\lesssim  \, \int_{P_1(Z_0)} \|\partial_t u\|_{L^{\infty}(P_1(Z_0))}^2\,\tau^2\,H_s(\tau,z)\,dz\,d\tau\,\\
        &+\int_{P_1(X_0,t_0)} \|\nabla u \|_{L^\infty(P_1(Z_0))}^2 |z|^2 \,H_s(\tau,z)\,dz\,d\tau+\int _{(P_1(Z_0))^c} H_s(\tau,z)\,d\tau\,dz\\
        &\lesssim \|\partial_t u\|^2_{L^\infty(P_1(Z_0))}\int_0^{1} \tau^{1-s}\,d\tau\,+\|\nabla u\|^2_{L^\infty(P_1(Z_0))}\int_0^1 \frac{\tau^{-s}}{\Gamma(-s)}d\tau+1\\
        &\lesssim \|\partial_t u\|^2_{L^{\infty}(P_1(Z_0))}+\|\nabla u\|^2_{L^\infty(P_1(Z_0))}+1,
    \end{split}
    \end{equation*}
    where we used that $\sup_{s\in (s_0,1)}|\Gamma(-s)|(1-s)<+\infty$. Thus, 
    \begin{equation*}
        \|\Xi_su\|_{L^\infty(P_{1/2}(Z_0))}^2\lesssim \|e(U)\|^2_{L^\infty(P_1^+(Z_0))}+1.
    \end{equation*}
    Using \eqref{eq: eq for |u|}, \eqref{eq: nonlocal GL flow} and \Cref{lem: clearing out}, we obtain
    \begin{equation*}
        (\partial_t-\Delta)^s \frac{|u|^2}{2}=\frac{1}{\eps^{2s}}(1-|u|^2)|u|^2-|\pg u|^2\geq \frac{1}{2\eps^{2s}}(1-|u|^2)-\|\pg u\|^2_{L^\infty(P_{1/2}(Z_0))}.
    \end{equation*}
    in $P_{1/2}(Z_0)$. Setting $\varphi=1-|u|^2$ we have 
    \begin{equation}
    \label{subsolution nonhom}
        \eps^{2s}(\partial_t-\Delta)^s\varphi+\varphi\leq 2\eps^{2s}\|\pg u\|^2_{L^\infty(P_{1/2}(Z_0))}.
    \end{equation}
    Noting that $\varphi\leq 1$, we can estimate $|\varphi|$ in $P_{1/2}(Z_0)$ by the sum of the constant solution $2\eps^{2s}\|\Xi_su\|_{L^\infty(P_{1/2}(Z_0))}^2$ and a subsolution to the problem 
    \begin{equation}
        \begin{cases}
            \eps^{2s}(\partial_t-\Delta)^s\psi+\psi\leq 0&\text{in }P_{1}(Z_0)\\
            \psi\leq 1&\text{in }P_{1}(Z_0)^c
        \end{cases}
    \end{equation}
    which admits a barrier of the form $\psi_\eps(x,t)\coloneqq h_\eps(t-t_0)+\eta_\eps(x-x_0)$, where $h_\eps$ and $\eta_\eps$ satisfy 
    \begin{equation*}
        \begin{cases}
            \eps^{2s}\partial_t^sh_\eps+h_\eps=0&t>-1\\
            h_\eps=1&t\leq -1,
        \end{cases}\qquad 
        \begin{cases}
            \eps^{2s}(-\Delta)^s\eta_\eps+\eta_\eps=0&x\in B_1\\
            \eta_\eps=1&x\in B_1^c,
        \end{cases}
    \end{equation*}
    It is possible to see that, for $t>-1$ and $x\in B_1$, 
    \begin{equation}
        |h_\eps(t)|\lesssim \frac{\eps^{2s}}{(t+1)^{s}},\quad |\eta_\eps(x)|\lesssim \frac{\eps^{2s}}{(1-|x|)^{2s}}.
    \end{equation}
    see \cite{Mainardi,Simon14} and \cite{RosOton-Serra14}, respectively.
    Thus 
    \begin{equation*}
        \frac{\psi_\eps(x,t)}{\eps^{2s}}\lesssim \frac{1}{(t-t_0+1)^s}+\frac{1}{(1-|x-x_0|)^{2s}}\lesssim 1\quad\text{in }P_{1/2}(
Z_0). 
    \end{equation*}
    Here, we used the comparison principle in \cite[Corollary 1.6]{StingaTorrea}. Thus, we get 
    \begin{equation*}
        \sup_{P_{1/2}(Z_0)}\frac{1-|u|^2}{\eps^{2s}}\leq 2\|\Xi_su\|_{L^\infty(P_{1/2}(Z_0))}^2+\sup_{P_{1/2}(Z_0)}\frac{\psi_\eps}{\eps^{2s}}\lesssim \|e(U)\|^2_{L^\infty(P_1^+(Z_0))}+1
    \end{equation*}
    concluding the proof.
\end{proof}

The following two results establish a certain control over weighted energy-like quantities by the monotone quantity $\Phi_\eps$ and the initial energy $E_0$, given by \eqref{eq: E0}.
\begin{lemma}\label{lem: L2 bound time}
	Let $U$ be a solution to \eqref{eq: GL flow}. Then, for any  $Z_0=(X_0,t_0)\in \R^m\times (0,+\infty)$ and $R>0$ we have 
	\begin{equation*}
		R^2\int_{T_R^+(Z_0)}z^a|\partial_tU|^2\G^s_{Z_0}\lesssim \Phi_\eps(U,Z_0,2R)+\int_{T_R^+(Z_0)}z^a|\nabla U|^2\G^s_{(X_0,t_0+R^2)}
	\end{equation*}
\end{lemma}
\begin{proof}
	Up to translating we can suppose $Z_0=(0,0)$. Let 
\begin{equation*}
	q(t)\coloneqq E_\eps (U(\cdot,t))= \int_{\R^{m+1}_+}z^a\frac{|\nabla U|^2}{2}\G^s+\int_{\R^m}W(u)\G^s,\quad F(t)\coloneqq tq(t)
\end{equation*}
Integrating the identity \eqref{eq: monotonicity remainder} over $(-4R^2,-R^2)$, we get 
\begin{equation*}
	\int_{-4R^2}^{-R^2}\frac{1}{4|t|}\int_{\R^{m+1}_+}z^a|2t\partial_tU+\nabla U\cdot X|^2\G^s\leq F(-R^2)-F(-4R^2)\leq|F(-4R^2)|
\end{equation*}
By the estimate 
\begin{equation*}
	|\partial_t U|^2\leq \frac{1}{2t^2}|2t\partial_tU+\nabla U\cdot X|^2+\frac{|X|^2}{2t^2}|\nabla U|^2
\end{equation*}
and monotonicity of $F$, we get, using also that $-4R^2\leq t\leq -R^2$ 
\begin{equation*}
	\int_{T^+_R}z^a|\partial_tU|^2\G^s\lesssim \frac{|F(-4R^2)|}{R^2}+\frac{1}{R^2}\int_{T^+_R}z^a\frac{|X|^2}{2|t|}|\nabla U|^2\G^s
\end{equation*}
Since $F$ is nondecreasing, for any $t\in (-16R^2,-4R^2)$ we have 
\begin{equation*}
	F(t)\leq F(-4R^2)\iff tq(t)\leq -4R^2q(-R^2)\iff q(t)\geq \frac{4R^2}{-t}q(-4R^2)
\end{equation*}
and integrating over $(-16R^2,-4R^2)$ we get
\begin{equation*}
	\Phi_\eps(U,0,2R)=\int_{-16R^2}^{-4R^2}q(t)dt\geq 4R^2q(-4R^2)\int_{-16R^2}^{-4R^2}\frac{dt}{-t}=(\log4)|F(-4R^2)|
\end{equation*}
thus we get 
\begin{equation*}
	R^2\int_{T^+_R}z^a|\partial_tU|^2\G^s\lesssim \Phi_\eps(U,0,2R)+\int_{T^+_R}z^a\frac{|X|^2}{|t|}|\nabla U|^2\G^s
\end{equation*}
The conclusion follows by observing that 
\begin{equation*}
	\frac{|X|^2}{|t|}\G^s(X,t)\lesssim \G^s_{(0,R^2)}(X,t)
\end{equation*}
for every $(X,t)\in T_R^+$.
\end{proof}

\begin{lemma}\label{lem: smallness}
	Let $U$ be a solution to \eqref{eq: GL flow}. Given $\mu>0$ there exists $\sigma>0$ depending on $\mu$ such that the following holds. Given a point $Z_0=(X_0,t_0)\in \R^m\times (0,+\infty)$ and $\rho>0$ such that $P^+_{2\rho}(Z_0)\subset P^+_\sigma$, then
	\begin{equation*}
		\rho^{2s-2-m}\int_{P_\rho^+(Z_0)}z^ae(U)(X,t)dXdt\lesssim  h\left(\Phi_\eps(U,0,\sigma)+\mu E_0\right)
	\end{equation*}
	where $h(t)\coloneqq t+\sqrt{t}$.
\end{lemma}
\begin{proof}
	We start by noting that 
	\begin{equation}\label{eq: shifted gaussian bound}
		\rho^{2s-2-m}\lesssim \G^s_{(X_0,t_0+2\rho^2)}(X,t),\quad\text{on }P_\rho^+(Z_0).
	\end{equation}
	Indeed, for $(X,t)\in P_{\rho}^+(Z_0)$, 
    \begin{equation*}
        |X-X_0|\leq \rho^2,\quad \rho^2\leq t_0+2\rho^2-t\leq 3\rho^2,
    \end{equation*}
    thus 
    \begin{equation*}
        \exp\left(-\frac{|X-X_0|}{4|t_0+2\rho^2-t|}\right)\geq e^{-1/4},\quad |t_0+2\rho^2-t|^{s-1-m/2}\geq (3\rho^2)^{s-1-m/2}
    \end{equation*}
    and hence \eqref{eq: shifted gaussian bound} holds. Thus 
    \begin{equation*}
    	\rho^{2s-2-m}\int_{P_\rho^+(Z_0)}z^ae(U)(X,t)dXdt\lesssim \int_{P_\rho^+(Z_0)}z^ae(U)(X,t)\G^s_{(X_0,t_0+2\rho^2)}dXdt.
    \end{equation*}
    Next, using that $P_\rho^+(Z_0)\subset T^+_\rho(t_0+2\rho^2)$ and the definition of $e(U)$
    \begin{equation*}
        \begin{split}
        	\int_{P_\rho^+(Z_0)}z^ae(U)(X,t)\G^s_{(X_0,t_0+2\rho^2)}dXdt&\leq \int_{T^+_\rho(t_0+2\rho^2)}z^a|\nabla U|^2(X,t)\G^s_{(X_0,t_0+2\rho^2)}dXdt\\
        	&+\int_{T^+_\rho(t_0+2\rho^2)}z^a|\partial_t U|(X,t)\G^s_{(X_0,t_0+2\rho^2)}dXdt
        \end{split}
    \end{equation*}
    Using H\"older inequality and \Cref{lem: L2 bound time}
    \begin{equation*}
    	\begin{split}
    		\left(\int_{T^+_\rho(t_0+2\rho^2)}z^a|\partial_t U|\G^s_{(X_0,t_0+2\rho^2)}\right)^2&\lesssim  \rho^2\int_{T^+_\rho(t_0+2\rho^2)}z^a|\partial_t U|^2\G^s_{(X_0,t_0+2\rho^2)}\\
    		&\lesssim \Phi_\eps (U,(X_0,t_0+2\rho^2),2\rho)\\
    		&+\int_{T^+_\rho(t_0+2\rho^2)}z^a|\nabla U|^2\G^s_{(X_0,t_0+3\rho^2)}
    	\end{split}
    \end{equation*}
    Now we invoke the inequalities 
    \begin{equation}\label{eq: rem_ineq}
    \begin{split}
\G^s_{(X_0,t_0+3\rho^2)}&\lesssim  \G^s_{(X_0,t_0+2\rho^2)}+\mu \sigma^{-2}\\
        \G^s_{(X_0,t_0+2\rho^2)}&\lesssim \G^s+\mu\sigma^{-2}
    \end{split}
    \end{equation} 
    which are standard and are essentially contained in \cite[Lemma 4.4]{HyderSegattiSireWang}. By \eqref{eq: rem_ineq} and monotonicity, 
    we find 
    \begin{equation*}
    	\begin{split}
    		\int_{T^+_\rho(t_0+2\rho^2)}z^a|\nabla U|^2\G^s_{(X_0,t_0+\ell\rho^2)}&\lesssim \int_{T_\sigma}z^a|\nabla U|^2\G^s+\mu \sigma^{-2}\int_{T_\sigma}z^a|\nabla U|^2\\
    		&\lesssim \Phi_\eps(U,0,\sigma)+\mu E_0.
    	\end{split}
    \end{equation*}
    Similarly, we obtain 
    \begin{equation*}
    	\Phi_\eps (U,(X_0,t_0+2\rho^2),2\rho)\lesssim \Phi_\eps(U,0,\sigma)+\mu E_0.
    \end{equation*}
    This concludes the proof. 
\end{proof}

\subsection{Schoen's argument and small energy regularity}\label{sec: Schoen}
With \Cref{lem: clearing out}, \Cref{lem: control pot by e} and \Cref{lem: smallness} we are ready to prove \Cref{thm: small energy GL}. 

\begin{proof}[Proof of \Cref{thm: small energy GL}]
    By scaling, we may suppose $R=1$. For any $\delta\in (0,\tfrac12)$ the map 
	\begin{equation*}
		[0,\delta]\ni \sigma\mapsto (\delta-\sigma)^2\max_{\overline P^+_\sigma}e(U)
	\end{equation*}
	is continuous. Suppose that it realises its maximum in $\sigma_0\in[0,\delta]$ and let $Z_0=(X_0,t_0)\in \overline{P}^+_{\sigma_0}$ be such that
	\begin{equation*}
		e_0\coloneqq \max_{\overline P^+_{\sigma_0}}e(U)=e(U)(Z_0).
	\end{equation*}
	We claim that $Z_0\in P_{\sigma_0}$. This follows from the fact that $U$ can be globally represented as the convolution of its trace with a unit mass kernel \cite[Theorem 1.7]{StingaTorrea}, thus $U$, $\nabla_xU$ and $\partial_tU$ all attain their maximum on $\{z=0\}$. Now, set $\rho_0=\frac{1}{2}(\delta-\sigma_0)$ and observe that 
	\begin{equation*}
		\sup_{B_{\rho_0}^+(Z_0)}e(U)\leq \sup_{B_{\rho_0+\sigma_0}^+(0)}e(U)\leq \left(\frac{\delta-\sigma_0}{\delta-(\rho_0+\sigma_0)}\right)^2e_0\leq 4e_0
	\end{equation*}
	Set $r_0=\sqrt{e_0}\rho_0$ and consider the map $V\in C^\infty(P_{r_0}^+)$ defined by 
	\begin{equation*}
		V(Y,\tau)\coloneqq U\left(\frac{Y}{\sqrt{e_0}}+X_0,\frac{\tau}{e_0}+t_0\right)
	\end{equation*}
    It is immediate to see that 
    \begin{equation*}
        e(V)(Y,\tau)=e_0^{-1}e(U)\left(\frac{Y}{\sqrt{e_0}}+X_0,\frac{\tau}{e_0}+t_0\right)
    \end{equation*}
    and thus that $e(V)(0,0)=1$.
    Moreover,
	\begin{equation*}
		\sup_{P_{r_0}^+(0)}e(V)=e_0^{-1}\sup_{P^+_{\rho_0}(Z_0)}e(U)\leq 4
	\end{equation*}
	hence $e(V)$ is uniformly bounded by an absolute constant in $P_{r_0}^+(0)$. Next, note that since $U$ solves \eqref{eq: GL flow}, then setting 
    \begin{equation*}
        Y=(y,\zeta)=\sqrt{e_0}(X-X_0),\quad \tau=e_0(t-t_0)
    \end{equation*}
    then $V$ solves 
    \begin{equation}\label{eq: eq V}
        \begin{cases}
            \zeta^{a}\partial_\tau V=\div_Y(\zeta^a\nabla_YV)&\text{in }P^+_{r_0}\\
            -\partial_\zeta^a\vert_{\zeta=0}V=(\eps\sqrt{e_0})^{-2s}(1-|V|^2)V&\text{on }\partial ^0P^+_{r_0}.
        \end{cases}
    \end{equation}
    We claim that $r_0\leq 1$. If this was false, then since $V$ solves \eqref{eq: eq V}, by \Cref{lem: Bochner} it holds
    \begin{equation*}
        \begin{cases}
            \zeta^a\partial_\tau e(V)-\div(\zeta^a\nabla e(V))\leq 0&\text{in }P^+_{r_0}\\
            -\partial_\zeta^ae(V)\leq \|(\eps\sqrt{e_0})^{-2s}(1-|v|^2)\|_{L^\infty(P_{r_0}^+(Z_0))}e(V)&\text{on }\partial ^0P^+_{r_0}.
        \end{cases}
    \end{equation*}
    By \Cref{lem: control pot by e}, we know that there exists $c_1>0$ depending only on the universal parameters such that 
    \begin{equation*}
        \left\Vert\frac{(1-|v|^2)}{(\eps\sqrt{e_0})^{2s}}\right\Vert_{L^\infty(P_{r_0/2}^+(Z_0))}\lesssim\, c_1(\|e(V)\|^2_{L^\infty(P_{r_0}^+(Z_0))}+1)\leq 17\,c_1.
    \end{equation*}
 Hence, we can apply Harnack's inequality \Cref{lem: Harnack} in $P_{1/2}^+\subset P_{r_0/2}^+ $, 
    \begin{equation*}
        1=e(V)(0,0)\lesssim \int_{P^+_{1/2}}\zeta^a e(V).
    \end{equation*}
    Changing variable, we find 
    \begin{equation*}
        1\lesssim \left(\frac1{\sqrt{e_0}}\right)^{2s-2-m}\int_{P_{1/2\sqrt{e_0}}^+(Z_0)}z^ae(U).
    \end{equation*}
    Next, we apply \Cref{lem: smallness} with $\sigma=\sigma_0$ and $\rho=\frac1{2\sqrt{e_0}}<\sigma<\delta$ to ensure that, given any $\mu>0$,
    \begin{equation*}
        1\lesssim \left(\eta_1+\mu E_0+\sqrt{\eta_1+\mu E_0}\right)
    \end{equation*}
    up to choosing $\delta$ small enough. This leads to a contradiction if we pick $\eta_1$ and $\mu$ sufficiently small, uniformly in $s\in (s_0,1)$. Thus $r_0\leq 1$ and hence 
\begin{equation*}
   \frac{\delta^2}4\sup_{P^+_{\delta/2}}e(U)\leq (\delta-\sigma_0)^2\sup_{P^+_{\sigma_0}}e(U)\leq 4\rho_0^2e_0\leq 4 
\end{equation*}
which concludes the proof.
\end{proof}

\section{Fractional harmonic maps}
We now turn our attention to the proof of \Cref{thm: small energy FHM}, which we divide in two steps. The first is a uniform small energy quantitative regularity for critical points satisfying a monotonicity formula under a qualitative smoothness assumption.

\begin{lemma}\label{lem: eps reg smooth FHM}
    There exist positive constants $\eps_1,\delta$ and $C$ depending only on $m,\ell$ and $s_0$ such that the following holds. Let $U\in C^\infty(\R^{m+1}_+,\R^\ell)$ be a solution to \eqref{eq: FHM br} for some $s\in (s_0,1)$ satisfying \eqref{eq: FHM monotonicity}. If, for some $x_0\in \R^{m}$ and $r>0$ 
    \begin{equation*}
        \Psi_s(U,x_0,r)\leq \eps_1
    \end{equation*}
    then 
    \begin{equation*}
        \sup_{B_{\delta r}} r^2|\nabla u|^2\leq C\delta^{-2}\   \Psi_s(U,x_0,r).
    \end{equation*}
\end{lemma}
\begin{proof}
    The proof follows essentially the same steps of \Cref{thm: small energy GL}, some of which are easier. First, we establish a Bochner inequality for the horizontal energy density 
    \begin{equation*}
        e(U)\coloneqq |\nabla_x U|^2.
    \end{equation*}
    Essentially the same calculations as \Cref{lem: Bochner} yield that for solutions to \eqref{eq: FHM br} it holds 
    \begin{equation*}
        \begin{cases}
            -\div(z^a\nabla e(U))\leq 0&\text{in }\R^{m+1}_+\\
            -\partial_z^a\vert_{z=0}e(U)\leq |d_su|^2e(U)&\text{on }\R^m.
        \end{cases}
    \end{equation*}
    This, plus the simple bound
    \begin{equation*}
        \sup_{B_{1/2}(x_0)}|d_su|^2\lesssim \|e(U)\|^2_{L^\infty(B^+_1)}+1
    \end{equation*}
    analogous to \Cref{lem: control pot by e}, yields 
    \begin{equation*}
        \begin{cases}
            -\div(z^a\nabla e(U))\leq 0&\text{in }B_{1/2}^+\\
            -\partial_z^a\vert_{z=0}e(U)\leq A (\|e(U)\|^2_{L^\infty(B^+_1)}+1)e(U)&\text{on }B_{1/2}.
        \end{cases}
    \end{equation*}
    Applying Schoen's argument with Harnack's inequality gives the conclusion, exactly as in \Cref{sec: Schoen}.
\end{proof}
Next, we employ \Cref{lem: eps reg smooth FHM} to show that below an energy threshold, critical points are smooth. To this end recall that thanks to the monotonicity formula every blow-up sequence $U_k\coloneqq U(X_0+\rho_k X)$ weakly converges to a map, called, tangent map as $\rho_k\to 0$, which is positively 0-homogeneous and solves the same equation, see \cite[Section 7.2]  {MillotPegonSchikorra}.
\begin{lemma}\label{lem: smooth}
    There exists $0<\eps_2<\eps_1$ depending only on $m,\ell$ and $s_0$ such that if $U\in H^1_a(\R^m,\R^\ell)$ is a weak solution to \eqref{eq: FHM br} for $s\in (s_0,1)$ and, for some $x_0\in R^m$ and $r>0$, 
    \begin{equation*}
        \Psi_s(U,x_0,r)\leq \eps_2
    \end{equation*}
    then $U\in C^\infty(B_{r/2}^+,\R^\ell)$.
\end{lemma}
\begin{proof}
We can suppose $r=1$ by scaling. Suppose by contradiction that there is one point $x_0\in B_{1/2}$ where $U$ is singular. Let $V$ be any tangent map of $U$ at $x_0$. By \cite[Theorem 5.1]{MillotPegonSchikorra}, $V$ is singular in the origin. Moreover, $V$ has at most $m-2$ directions of translation invariance \cite[Lemma 7.12]{MillotPegonSchikorra}, in the sense that there exists a homogeneous map $W\colon\R^{n+1}_+\to \Sp^\ell$ such that 
\begin{equation*}
    V(x,z)=W(x',z),\quad x=(x',x'')\in \R^{n}\times \R^{m-n}
\end{equation*}
where $2\leq n\leq m$. Note that $W$ is smooth in $B_{2}\setminus B_{1/2}$. Thus, for any $\eta>0$ we can pick $\eps_2$ sufficiently small such that, using a covering argument and \Cref{lem: eps reg smooth FHM} 
\begin{equation*}
    \|w-q\|_{L^\infty(B_{3/2}\setminus B_{3/4})}<\eta.
\end{equation*}
where $w=\Tr W$. Hence, looking at the link over $\Sp^{n-1}$
\begin{equation}\label{eq: small ball image}
    w(\Sp^{n-1})\subset B_\eta(q).
\end{equation}
The $s$-harmonicity on 0-homogeneous maps transfers to the link with an averaged kernel, in the sense that, setting $w(r,\theta)=g(\theta)$
\begin{equation*}
    \begin{split}
        \int_{\R^n}\frac{w(x)-w(y)}{|x-y|^{n+2s}}dy&= \int_{\Sp^{n-1}}(g(\theta)-g(\phi))K_s(\theta,\phi)dV_{\Sp^{n-1}}(\phi)\eqqcolon (-\Delta)^s_{\Sp^{n-1}}g
    \end{split}
\end{equation*}
where 
\begin{equation}\label{eq: Ks}
    K_s(\theta,\phi)\coloneqq \int_0^\infty\frac{t^{n-1}}{(1+t^2-2t\theta\cdot\phi)^\frac{n+2s}{2}}dt
\end{equation}
and thus we obtain 
\begin{equation}\label{eq: orth}
    (-\Delta)^s_{\Sp^{n-1}}g\perp T_{g}\Sp^{n-1}.
\end{equation}
Let now $\rho\colon B_\eta (q)\cap \Sp^{\ell}\to \R$ be the geodetic distance to the point $q\in\Sp^\ell$. For $\eta>0$ sufficiently small, this is a strictly convex function and can be extended (without renaming) as a strictly convex function to the whole $B_\eta(q)$. Let now $f\coloneqq \rho\circ w\colon \Sp^{n-1}\to\R$, which is well defined by \eqref{eq: small ball image}. By strict convexity, there exists a constant $c>0$ such that 
\begin{equation*}
	\rho(w(x))-\rho(w(y))\leq \nabla\rho(w(x))\cdot (w(x)-w(y))-c|w(x)-w(y)|^2.
\end{equation*}
Passing to the angular function $w(x)=g(\theta)$ and integrating against \eqref{eq: Ks}, we get 
\begin{equation}\label{eq: KW arg}
     (-\Delta)^s_{\Sp^{n-1}}f(\theta)\leq \nabla \rho (g(\theta))\cdot (-\Delta)^s_{\Sp^{n-1}}g(\theta)-c\int_{\Sp^{n-1}}|g(\theta)-g(\phi)|^2K_s(\theta,\phi)dV_{\Sp^{n-1}}(\phi)
\end{equation}
and since $\nabla\rho\circ g\in T_u\Sp^{n-1}$ we obtain, using \eqref{eq: orth}, that $ (-\Delta)^s_{\Sp^{n-1}}f\leq 0$. Thus, $f$ is constant and by \eqref{eq: KW arg} $g$ is constant too. We reached a contradiction and the proof is concluded.
\end{proof}

\begin{proof}[Proof of \Cref{thm: small energy FHM}]
    It is sufficient to combine \Cref{lem: eps reg smooth FHM} and \Cref{lem: smooth}.
\end{proof}

\appendix
\section{Uniform Harnack inequality}

The following result is a uniform-in-$s$ version of the $L^1$-to-$L^\infty$ Harnack inequality for positive subsolutions of the problem 
\begin{equation}\label{eq: subsol}
	\begin{cases}
		z^a\partial_tf-\div(z^a\nabla f)\leq 0&\text{in } P_1^+\\
		-\partial_z^a\vert_{z=0}f\leq C_0f&\text{on } P_1
	\end{cases}
\end{equation}
where, for $r>0$, $P_r^+\coloneqq P_r\times (0,r)$\footnote{We work with cylindrical extensions rather than half-ball extensions just out of convenience, the results remain true up to absolute constants.}.
\begin{lemma}[Harnack inequality]\label{lem: Harnack}
	Let $f\geq 0$ solve \eqref{eq: subsol} for some $s\in (s_0,1)$. Then, there is a constant $C>0$ depending only on $m,s_0$ and $C_0$ such that 
	\begin{equation*}
		f(0,0)\leq C\int_{P_1^+}z^a f(X,t)dXdt.
	\end{equation*}
	\begin{proof}
		The statement follows by a standard Moser iteration scheme with an extra step to control the boundary term. 
        We sketch the proof for completeness. 
		
		\medskip
		
		Let $0<\rho<R\le 1$, let $\eta\in C_c^\infty(P_R)$ satisfy
\begin{equation*}
	0\le \eta\le 1,\qquad \eta\equiv 1\ \text{on }P_\rho,\qquad
|\nabla\eta|+|\partial_t\eta|^{1/2}\lesssim (R-\rho)^{-1},
\end{equation*}
and fix $p\ge 1$. Testing \eqref{eq: subsol} with $\phi=\eta^2 f^{p-1}$ gives the energy estimate
\begin{equation}\label{eq:caccioppoli-pre}
\begin{split}
	\sup_{t\in(-R^2,R^2)}\int_{B_R\times(0,R)} z^a \eta^2 f^p
&+\int_{P_R^+} z^a |\nabla(\eta f^{p/2})|^2
\\
&\lesssim
\frac{p^2}{(R-\rho)^2}\int_{P_R^+} z^a f^p
+C_0 p\int_{P_R} (\eta f^{p/2})^2 .
\end{split}
\end{equation}
The last term can be reabsorbed in the right-hand side by the weighted trace inequality
\begin{equation*}
    \int_{B_R}w^2\leq C\eps R^{2s} \int_{B_R^+}z^a|\nabla w|^2+C\eps^{-1}R^{-2(1-s)}\int_{B_R^+}z^aw^2
\end{equation*}
whose constant we claim to be uniform in $s\in(s_0,1)$.
Accepting this for the moment, the conclusion follows by integrating over time and choosing $\eps>0$ sufficiently small.
Thus, denoting $F\coloneqq \eta f^{p/2}$, we obtain 
\begin{equation}\label{eq:caccioppoli}
	\begin{split}
	&\sup_{t\in(-R^2,R^2)}\int_{B_R\times(0,R)} z^a |F|^2
+\int_{P_R^+} z^a |\nabla F|^2\lesssim
\frac{p^2}{(R-\rho)^2}\int_{P_R^+} z^a f^p.
\end{split}
\end{equation}
Now, the weighted parabolic Sobolev inequality \cite[Lemma 1.2]{Chiarenza} states that there exists $\chi>1$ such that, denoting 
\begin{equation*}
    d\mu_a\coloneqq \frac{\delta_s}{2}|z|^adX,\quad\text{and}\quad  \fint _{\Omega}f \coloneqq \frac{1}{\int_\Omega d\mu_a}\int_\Omega f d\mu_a
\end{equation*}
we have
\begin{equation}\label{eq:par-sob}
\fint_{P_\rho^+} |V|^{2\chi}
\leq
C\rho^2\left(
\sup_{t\in(-\rho^2,\rho^2)}\fint_{B_\rho^+}|V|^2
\right)^{\chi-1}
\left(
\fint_{P_\rho^+} |\nabla V|^2
\right).
\end{equation}
again with a constant uniform in $s\in (s_0,1)$.
Applying \eqref{eq:par-sob} to $F$ and using \eqref{eq:caccioppoli}, one gets
\begin{equation*}
	\|f\|_{L^{p\chi}(P_\rho^+,d\mu_a)}
\le
\left(\frac{C\,p^2}{(R-\rho)^2}\right)^{1/p}
\|f\|_{L^p(P_R^+,d\mu_a)}.
\end{equation*}
Iterating with radii $\rho_k=\frac12+2^{-k-1}$, $R_k=\frac12+2^{-k}$ and powers $p_{k+1}=\chi p_k$ (with $p_0=1$) one gets 
\begin{equation*}
	\|f\|_{L^\infty(P_{1/2}^+)}
\le C \|f\|_{L^1(P_1^+,d\mu_a)}
\end{equation*}
with a constant uniform in $s\in (s_0,1)$, up to proving the claim. By the fundamental theorem of calculus and Cauchy--Schwarz 
\begin{equation*}
    |w(x,0)|^2\leq 2|w(x,z)|^2+\frac{2}{1-a}z^{1-a}\int_0^zt^a|\partial_t w(x,t)|^2dt.
\end{equation*}
Multiplying by $z^a$, integrating over $z\in (0,\eta)$ and using that 
\begin{equation*}
    \int_0^\eta z\int_0^zg(t)dtdz=\int_0^\eta \frac{\eta^2-t^2}{2}g(t)dt\leq \frac{\eta^2}{2}\int_0^\eta g(t)dt
\end{equation*}
We get 
\begin{equation*}
    |w(x,0)|^2\leq \frac{2(1+a)}{\eta^{1+a}}\int_0^\eta z^a|w(x,z)|^2dz+\frac{1+a}{1-a}\eta^{1-a}\int_0^\eta z^a|\partial_zw(x,z)|^2dz.
\end{equation*}
Integrating over $B_R$ and setting $\eta=R\eps^\frac{1}{2s}$ proves the claim, once we note that $\frac{1+a}{1-a}=\frac{1-s}s\lesssim 1$ and that $\delta_s/(1+a)\to 1$ as $s\nearrow 1$.
\end{proof}
\end{lemma}

\medskip

\noindent\textbf{Acknowledgments.}
The authors are grateful to J.~ Serra for his support during the preparation of this paper. 
MB was supported by the European Research Council under Grant Agreement No 948029. Part of this work was carried out while GC was visiting ETH Z\"urich in Spring 2026.
%Thanks Joaquim 

\bibliography{ref}

@article {Chiarenza,
    AUTHOR = {Chiarenza, Filippo and Serapioni, Raul},
     TITLE = {A remark on a {H}arnack inequality for degenerate parabolic
              equations},
   JOURNAL = {Rend. Sem. Mat. Univ. Padova},
  FJOURNAL = {Rendiconti del Seminario Matematico della Universit\`a di
              Padova. The Mathematical Journal of the University of Padova},
    VOLUME = {73},
      YEAR = {1985},
     PAGES = {179--190},
      ISSN = {0041-8994},
   MRCLASS = {35K65 (35B45)},
  MRNUMBER = {799906},
MRREVIEWER = {Rouben\ Rostamian},
       URL = {http://www.numdam.org/item?id=RSMUP_1985__73__179_0},
}

@article {StingaTorrea,
    AUTHOR = {Stinga, Pablo Ra\'{u}l and Torrea, Jos\'{e} L.},
     TITLE = {Regularity theory and extension problem for fractional
              nonlocal parabolic equations and the master equation},
   JOURNAL = {SIAM J. Math. Anal.},
  FJOURNAL = {SIAM Journal on Mathematical Analysis},
    VOLUME = {49},
      YEAR = {2017},
    NUMBER = {5},
     PAGES = {3893--3924},
      ISSN = {0036-1410,1095-7154},
   MRCLASS = {35R11 (26A33 35B65 35R09 47G20 58J35)},
  MRNUMBER = {3709888},
MRREVIEWER = {Kai\ Diethelm},
       DOI = {10.1137/16M1104317},
       URL = {https://doi.org/10.1137/16M1104317},
}

@article {DaLioRiv1,
    AUTHOR = {Da Lio, Francesca and Rivi\`ere, Tristan},
     TITLE = {Three-term commutator estimates and the regularity of
              {$\frac12$}-harmonic maps into spheres},
   JOURNAL = {Anal. PDE},
  FJOURNAL = {Analysis \& PDE},
    VOLUME = {4},
      YEAR = {2011},
    NUMBER = {1},
     PAGES = {149--190},
      ISSN = {2157-5045,1948-206X},
   MRCLASS = {58E20 (35B65 35J20 35J62 35S05)},
  MRNUMBER = {2783309},
MRREVIEWER = {Fr\'{e}d\'{e}ric\ Robert},
       DOI = {10.2140/apde.2011.4.149},
       URL = {https://doi.org/10.2140/apde.2011.4.149},
}

@article {DaLioRiv2,
    AUTHOR = {Da Lio, Francesca and Rivi\`ere, Tristan},
     TITLE = {Sub-criticality of non-local {S}chr\"{o}dinger systems with
              antisymmetric potentials and applications to half-harmonic
              maps},
   JOURNAL = {Adv. Math.},
  FJOURNAL = {Advances in Mathematics},
    VOLUME = {227},
      YEAR = {2011},
    NUMBER = {3},
     PAGES = {1300--1348},
      ISSN = {0001-8708,1090-2082},
   MRCLASS = {58E20 (35B65 35R11)},
  MRNUMBER = {2799607},
MRREVIEWER = {Andreas\ Gastel},
       DOI = {10.1016/j.aim.2011.03.011},
       URL = {https://doi.org/10.1016/j.aim.2011.03.011},
}

@article {FraserSchoen19,
    AUTHOR = {Fraser, Ailana and Schoen, Richard},
     TITLE = {Shape optimization for the {S}teklov problem in higher
              dimensions},
   JOURNAL = {Adv. Math.},
  FJOURNAL = {Advances in Mathematics},
    VOLUME = {348},
      YEAR = {2019},
     PAGES = {146--162},
      ISSN = {0001-8708,1090-2082},
   MRCLASS = {58J50 (35J25 35P15 35R01)},
  MRNUMBER = {3925929},
MRREVIEWER = {Alexander\ G.\ Losev},
       DOI = {10.1016/j.aim.2019.03.011},
       URL = {https://doi.org/10.1016/j.aim.2019.03.011},
}

@article {FraserSchoen16,
    AUTHOR = {Fraser, Ailana and Schoen, Richard},
     TITLE = {Sharp eigenvalue bounds and minimal surfaces in the ball},
   JOURNAL = {Invent. Math.},
  FJOURNAL = {Inventiones Mathematicae},
    VOLUME = {203},
      YEAR = {2016},
    NUMBER = {3},
     PAGES = {823--890},
      ISSN = {0020-9910,1432-1297},
   MRCLASS = {58C40 (35P15 53A10)},
  MRNUMBER = {3461367},
MRREVIEWER = {Isabel\ M. C. Salavessa},
       DOI = {10.1007/s00222-015-0604-x},
       URL = {https://doi.org/10.1007/s00222-015-0604-x},
}

@article {MillotPegon,
    AUTHOR = {Millot, Vincent and Pegon, Marc},
     TITLE = {Minimizing 1/2-harmonic maps into spheres},
   JOURNAL = {Calc. Var. Partial Differential Equations},
  FJOURNAL = {Calculus of Variations and Partial Differential Equations},
    VOLUME = {59},
      YEAR = {2020},
    NUMBER = {2},
     PAGES = {Paper No. 55, 37},
      ISSN = {0944-2669,1432-0835},
   MRCLASS = {35R11 (35B65 35J47 58E20)},
  MRNUMBER = {4066534},
MRREVIEWER = {Jens\ Wirth},
       DOI = {10.1007/s00526-020-1704-z},
       URL = {https://doi.org/10.1007/s00526-020-1704-z},
}

@article {MillotPegonSchikorra,
    AUTHOR = {Millot, Vincent and Pegon, Marc and Schikorra, Armin},
     TITLE = {Partial regularity for fractional harmonic maps into spheres},
   JOURNAL = {Arch. Ration. Mech. Anal.},
  FJOURNAL = {Archive for Rational Mechanics and Analysis},
    VOLUME = {242},
      YEAR = {2021},
    NUMBER = {2},
     PAGES = {747--825},
      ISSN = {0003-9527,1432-0673},
   MRCLASS = {58E20 (35R11 49N60)},
  MRNUMBER = {4331016},
MRREVIEWER = {Fr\'{e}d\'{e}ric\ Robert},
       DOI = {10.1007/s00205-021-01693-w},
       URL = {https://doi.org/10.1007/s00205-021-01693-w},
}

@article {MillotSire,
    AUTHOR = {Millot, Vincent and Sire, Yannick},
     TITLE = {On a fractional {G}inzburg-{L}andau equation and 1/2-harmonic
              maps into spheres},
   JOURNAL = {Arch. Ration. Mech. Anal.},
  FJOURNAL = {Archive for Rational Mechanics and Analysis},
    VOLUME = {215},
      YEAR = {2015},
    NUMBER = {1},
     PAGES = {125--210},
      ISSN = {0003-9527,1432-0673},
   MRCLASS = {35R11 (35B25 35Q56)},
  MRNUMBER = {3296146},
MRREVIEWER = {Xingbin\ Pan},
       DOI = {10.1007/s00205-014-0776-3},
       URL = {https://doi.org/10.1007/s00205-014-0776-3},
}

@article {MillotSireYu,
    AUTHOR = {Millot, Vincent and Sire, Yannick and Yu, Hui},
     TITLE = {Minimizing fractional harmonic maps on the real line in the
              supercritical regime},
   JOURNAL = {Discrete Contin. Dyn. Syst.},
  FJOURNAL = {Discrete and Continuous Dynamical Systems},
    VOLUME = {38},
      YEAR = {2018},
    NUMBER = {12},
     PAGES = {6195--6214},
      ISSN = {1078-0947,1553-5231},
   MRCLASS = {58E20 (35R11)},
  MRNUMBER = {3917807},
MRREVIEWER = {Christopher\ Steven\ Goodrich},
       DOI = {10.3934/dcds.2018266},
       URL = {https://doi.org/10.3934/dcds.2018266},
}

@article {MillotSireWang,
    AUTHOR = {Millot, Vincent and Sire, Yannick and Wang, Kelei},
     TITLE = {Asymptotics for the fractional {A}llen-{C}ahn equation and
              stationary nonlocal minimal surfaces},
   JOURNAL = {Arch. Ration. Mech. Anal.},
  FJOURNAL = {Archive for Rational Mechanics and Analysis},
    VOLUME = {231},
      YEAR = {2019},
    NUMBER = {2},
     PAGES = {1129--1216},
      ISSN = {0003-9527,1432-0673},
   MRCLASS = {35R11 (26A33 35B25 49Q05)},
  MRNUMBER = {3900821},
MRREVIEWER = {Krzysztof\ Rogowski},
       DOI = {10.1007/s00205-018-1296-3},
       URL = {https://doi.org/10.1007/s00205-018-1296-3},
}

@article {DuzaarSteffen,
    AUTHOR = {Duzaar, Frank and Steffen, Klaus},
     TITLE = {A partial regularity theorem for harmonic maps at a free
              boundary},
   JOURNAL = {Asymptotic Anal.},
  FJOURNAL = {Asymptotic Analysis},
    VOLUME = {2},
      YEAR = {1989},
    NUMBER = {4},
     PAGES = {299--343},
      ISSN = {0921-7134},
   MRCLASS = {58E20},
  MRNUMBER = {1030353},
MRREVIEWER = {Helmut\ Kaul},
}

@article {Hardt-Lin,
    AUTHOR = {Hardt, Robert and Lin, Fang-Hua},
     TITLE = {Partially constrained boundary conditions with energy
              minimizing mappings},
   JOURNAL = {Comm. Pure Appl. Math.},
  FJOURNAL = {Communications on Pure and Applied Mathematics},
    VOLUME = {42},
      YEAR = {1989},
    NUMBER = {3},
     PAGES = {309--334},
      ISSN = {0010-3640,1097-0312},
   MRCLASS = {49F10 (58E15)},
  MRNUMBER = {982353},
MRREVIEWER = {Martin\ Fuchs},
       DOI = {10.1002/cpa.3160420306},
       URL = {https://doi.org/10.1002/cpa.3160420306},
}

@article {SireWeiZheng,
    AUTHOR = {Sire, Yannick and Wei, Juncheng and Zheng, Youquan},
     TITLE = {Infinite time blow-up for half-harmonic map flow from {$\Bbb
              R$} into {$\Bbb S^1$}},
   JOURNAL = {Amer. J. Math.},
  FJOURNAL = {American Journal of Mathematics},
    VOLUME = {143},
      YEAR = {2021},
    NUMBER = {4},
     PAGES = {1261--1335},
      ISSN = {0002-9327,1080-6377},
   MRCLASS = {53E99 (58E20)},
  MRNUMBER = {4291253},
MRREVIEWER = {Abimbola\ Abolarinwa},
       DOI = {10.1353/ajm.2021.0031},
       URL = {https://doi.org/10.1353/ajm.2021.0031},
}

@article {HyderSegattiSireWang,
    AUTHOR = {Hyder, Ali and Segatti, Antonio and Sire, Yannick and Wang,
              Changyou},
     TITLE = {Partial regularity of the heat flow of half-harmonic maps and
              applications to harmonic maps with free boundary},
   JOURNAL = {Comm. Partial Differential Equations},
  FJOURNAL = {Communications in Partial Differential Equations},
    VOLUME = {47},
      YEAR = {2022},
    NUMBER = {9},
     PAGES = {1845--1882},
      ISSN = {0360-5302,1532-4133},
   MRCLASS = {35K05},
  MRNUMBER = {4472924},
       DOI = {10.1080/03605302.2022.2091453},
       URL = {https://doi.org/10.1080/03605302.2022.2091453},
}

@article {CaffarelliSilvestre,
    AUTHOR = {Caffarelli, Luis and Silvestre, Luis},
     TITLE = {An extension problem related to the fractional {L}aplacian},
   JOURNAL = {Comm. Partial Differential Equations},
  FJOURNAL = {Communications in Partial Differential Equations},
    VOLUME = {32},
      YEAR = {2007},
    NUMBER = {7-9},
     PAGES = {1245--1260},
      ISSN = {0360-5302,1532-4133},
   MRCLASS = {35J70},
  MRNUMBER = {2354493},
MRREVIEWER = {Francesco\ Petitta},
       DOI = {10.1080/03605300600987306},
       URL = {https://doi.org/10.1080/03605300600987306},
}

@article {SerraSeMA,
    AUTHOR = {Serra, Joaquim},
     TITLE = {Nonlocal minimal surfaces: recent developments, applications,
              and future directions},
   JOURNAL = {SeMA J.},
  FJOURNAL = {SeMA Journal. Boletin de la Sociedad Espa\~{n}ola de
              Matem\'{a}tica Aplicada},
    VOLUME = {81},
      YEAR = {2024},
    NUMBER = {2},
     PAGES = {165--191},
      ISSN = {2254-3902,2281-7875},
   MRCLASS = {49-02 (49Q05 53A10)},
  MRNUMBER = {4743530},
       DOI = {10.1007/s40324-023-00345-1},
       URL = {https://doi.org/10.1007/s40324-023-00345-1},
}

@article{Badran26,
author = {Badran, Marco},
title = {Harmonic maps to the circle with higher dimensional singular set},
journal = {Proceedings of the London Mathematical Society},
volume = {132},
number = {3},
pages = {e70135},
doi = {https://doi.org/10.1112/plms.70135},
url = {https://londmathsoc.onlinelibrary.wiley.com/doi/abs/10.1112/plms.70135},
eprint = {https://londmathsoc.onlinelibrary.wiley.com/doi/pdf/10.1112/plms.70135},
year = {2026}
}

@misc{CaselliFregugliaPicenni2024,
      title={A nonlocal approximation of the area in codimension two}, 
      author={Caselli, Michele and Freguglia, Mattia and Picenni, Nicola},
      year={2024},
      eprint={2406.13696},
      archivePrefix={arXiv},
      primaryClass={math.DG},
      url={https://arxiv.org/abs/2406.13696}, 
}

@misc{CaselliFregugliaPicenni2026,
      title={Another look at a notion of fractional mass in codimension two}, 
      author={Caselli, Michele and Freguglia, Mattia and Picenni, Nicola},
      year={forthcoming, 2026}, 
}

@article {Struwe88,
    AUTHOR = {Struwe, Michael},
     TITLE = {On the evolution of harmonic maps in higher dimensions},
   JOURNAL = {J. Differential Geom.},
  FJOURNAL = {Journal of Differential Geometry},
    VOLUME = {28},
      YEAR = {1988},
    NUMBER = {3},
     PAGES = {485--502},
      ISSN = {0022-040X,1945-743X},
   MRCLASS = {58E20 (35Bxx 58G11)},
  MRNUMBER = {965226},
MRREVIEWER = {J.\ Eells},
       URL = {http://projecteuclid.org/euclid.jdg/1214442475},
}

@incollection {Schoen84,
    AUTHOR = {Schoen, Richard M.},
     TITLE = {Analytic aspects of the harmonic map problem},
 BOOKTITLE = {Seminar on nonlinear partial differential equations
              ({B}erkeley, {C}alif., 1983)},
    SERIES = {Math. Sci. Res. Inst. Publ.},
    VOLUME = {2},
     PAGES = {321--358},
 PUBLISHER = {Springer, New York},
      YEAR = {1984},
      ISBN = {0-387-96079-1},
   MRCLASS = {58E20},
  MRNUMBER = {765241},
MRREVIEWER = {Helmut\ Kaul},
       DOI = {10.1007/978-1-4612-1110-5\{_}17}

@article {Struwe24,
    AUTHOR = {Struwe, Michael},
     TITLE = {Plateau flow or the heat flow for half-harmonic maps},
   JOURNAL = {Anal. PDE},
  FJOURNAL = {Analysis \& PDE},
    VOLUME = {17},
      YEAR = {2024},
    NUMBER = {4},
     PAGES = {1397--1438},
      ISSN = {2157-5045,1948-206X},
   MRCLASS = {35K55 (35K65 53A10 53E99)},
  MRNUMBER = {4746874},
MRREVIEWER = {David\ James\ Hartley},
       DOI = {10.2140/apde.2024.17.1397},
       URL = {https://doi.org/10.2140/apde.2024.17.1397},
}

@article {Roberts,
    AUTHOR = {Roberts, James},
     TITLE = {A regularity theory for intrinsic minimising fractional
              harmonic maps},
   JOURNAL = {Calc. Var. Partial Differential Equations},
  FJOURNAL = {Calculus of Variations and Partial Differential Equations},
    VOLUME = {57},
      YEAR = {2018},
    NUMBER = {4},
     PAGES = {Paper No. 109, 68},
      ISSN = {0944-2669,1432-0835},
   MRCLASS = {58E20 (35B05 35B65 35J70 35R11 35R35)},
  MRNUMBER = {3817790},
MRREVIEWER = {Armin\ Schikorra},
       DOI = {10.1007/s00526-018-1384-0},
       URL = {https://doi.org/10.1007/s00526-018-1384-0},
}

@article {Moser2011,
    AUTHOR = {Moser, Roger},
     TITLE = {Intrinsic semiharmonic maps},
   JOURNAL = {J. Geom. Anal.},
  FJOURNAL = {Journal of Geometric Analysis},
    VOLUME = {21},
      YEAR = {2011},
    NUMBER = {3},
     PAGES = {588--598},
      ISSN = {1050-6926,1559-002X},
   MRCLASS = {58E20 (35B38 35J50 35R11)},
  MRNUMBER = {2810844},
MRREVIEWER = {Christoph\ Scheven},
       DOI = {10.1007/s12220-010-9159-7},
       URL = {https://doi.org/10.1007/s12220-010-9159-7},
}

@article {RosOton-Serra14,
    AUTHOR = {Ros-Oton, Xavier and Serra, Joaquim},
     TITLE = {The {D}irichlet problem for the fractional {L}aplacian:
              regularity up to the boundary},
   JOURNAL = {J. Math. Pures Appl. (9)},
  FJOURNAL = {Journal de Math\'{e}matiques Pures et Appliqu\'{e}es.
              Neuvi\`eme S\'{e}rie},
    VOLUME = {101},
      YEAR = {2014},
    NUMBER = {3},
     PAGES = {275--302},
      ISSN = {0021-7824,1776-3371},
   MRCLASS = {35R11 (35B65)},
  MRNUMBER = {3168912},
MRREVIEWER = {Kai\ Diethelm},
       DOI = {10.1016/j.matpur.2013.06.003},
       URL = {https://doi.org/10.1016/j.matpur.2013.06.003},
}

@article {Simon14,
    AUTHOR = {Simon, Thomas},
     TITLE = {Comparing {F}r\'{e}chet and positive stable laws},
   JOURNAL = {Electron. J. Probab.},
  FJOURNAL = {Electronic Journal of Probability},
    VOLUME = {19},
      YEAR = {2014},
     PAGES = {no. 16, 25},
      ISSN = {1083-6489},
   MRCLASS = {60E05 (33E12 60E15 60G52 62E15)},
  MRNUMBER = {3164769},
       DOI = {10.1214/EJP.v19-3058},
       URL = {https://doi.org/10.1214/EJP.v19-3058},
}

@article {Mainardi,
    AUTHOR = {Mainardi, Francesco},
     TITLE = {On some properties of the {M}ittag-{L}effler function
              {$E_\alpha(-t^\alpha)$}, completely monotone for {$t>0$}
              with {$0<\alpha<1$}},
   JOURNAL = {Discrete Contin. Dyn. Syst. Ser. B},
  FJOURNAL = {Discrete and Continuous Dynamical Systems. Series B. A Journal
              Bridging Mathematics and Sciences},
    VOLUME = {19},
      YEAR = {2014},
    NUMBER = {7},
     PAGES = {2267--2278},
      ISSN = {1531-3492,1553-524X},
   MRCLASS = {33E12 (26A33)},
  MRNUMBER = {3253257},
       DOI = {10.3934/dcdsb.2014.19.2267},
       URL = {https://doi.org/10.3934/dcdsb.2014.19.2267},
}

@misc{KimNowakSire,
      title={Fine regularity of fractional harmonic maps and applications}, 
      author={Kyeongbae Kim and Simon Nowak and Yannick Sire},
      year={2026},
      eprint={2602.15715},
      archivePrefix={arXiv},
      primaryClass={math.AP},
      url={https://arxiv.org/abs/2602.15715}, 
}

@misc{GengWang,
      title={On the existence of partially smooth solutions to the heat flow of $s$-harmonic maps}, 
      author={Geng, Zhiyuan and Wang, Changyou},
      year={forthcoming, 2026}, 
}
\bibliographystyle{siam}

\end{document}